\documentclass[a4paper,12pt]{article}
\usepackage{amsthm}
\usepackage{url}
\usepackage{amsmath}
\usepackage{color}
\usepackage{amsfonts}
\usepackage{graphicx}
\usepackage{nameref}
\usepackage{varioref}
\usepackage{cleveref}
\usepackage{spverbatim}
\usepackage{listings}
\usepackage{caption}
\usepackage{subcaption}
\usepackage{pythonhighlight}

\newtheorem{prop}{Proposition}
	\newtheorem{definition}{Definition}

	\theoremstyle{definition}
	\newtheorem{example}{Example}
	\newtheorem{remark}{Remark}

\begin{document}
\fontsize{12pt}{12pt}
\usefont{OT1}{cmr}{m}{n}
\begin{center}
\textbf{Solver-free optimal control for Linear Dynamical Switched System\\ by means of Geometric Algebra }
\end{center}

\begin{center}
Anna Derevianko, Petr Va\v s\'ik
\end{center}

\noindent
\textbf{Abstract.} 	We design an algorithm for control of a linear switched system by means of Geometric Algebra. More precisely, we develop a switching path searching algorithm for a two-dimensional linear dynamical switched system with non-singular matrix whose integral curves are formed by two sets of centralised ellipses. Then it is natural to represent them as elements of Geometric Algebra for Conics (GAC) and construct the switching path by calculating the switching points, i.e. intersections and contact points. For this, we use symbolic algebra operations, more precisely the wedge and inner products, that are realisable by sums of products in the coordinate form. Therefore, no numerical solver to the system of equations is needed. Indeed, the only operation that may bring in an inaccuracy is a vector normalisation, i.e. square root calculation. The resulting switching path is formed by pieces of ellipses that are chosen respectively from the two sets of integral curves. The switching points are either intersections in the first or final step of our algorithm, or contact points. This choice guarantees optimality of the switching path with respect to the number of switches. On two examples we demonstrate the search for conics' intersections and, consequently, we describe a construction of a switching path in both cases.  

\vspace*{12pt} 
\vspace*{12pt}


\fontsize{12pt}{12pt}
\usefont{OT1}{cmr}{m}{n}

	\section*{Introduction}
	Switched systems form a special case of hybrid dynamical systems with discrete and continuous dynamics.
	They are widely applied when a real system cannot be described by one single model. Numerous examples are given by engineering systems of electronics, power systems, traffic control and others. 
	Since the 1990s, research of switched systems stability has become very popular, see e.g.
	\cite{Vid,Sta}.
	The particular case of linear switched systems can be found in 
	\cite{Col}. 
	More recent literature about switched systems is represented by the works of Patrizio Colaneri
	\cite{Col}, Yuan Lin, Yuan Sun-Ge Wang,  and Jiang-Wang \cite{Lin}, Zhong-Ping, Yuan Wang \cite{Ji}; 
	the question of stability remains widely studied today.
	
	In the sequel, we use the power of geometric algebra (GA) to control a $2\times 2$ linear dynamical switched system with non-singular matrices where we exploit the fact that their control paths form a set of centralised ellipses and for a switched system control it is enough to find the point of a switch, i.e. the intersection of two ellipses. Classically, this leads to a system of quadratic equations which is simple to solve numerically but a certain inaccuracy is involved, i.e. a wrong control path is chosen. This error increases with the increasing number of switches due to the inaccurate initial conditions. Therefore we introduce an algorithm for searching ellipses' intersections with no solver needed, which together with a straightforward way of conic scaling leads to a control with minimal number of switches. In this sense, our control is optimal.
	
	If we restrict to the above described class of switched systems, we can design their (optimal) control by means of Geometric Algebra for Conics (GAC). This is an efficient geometric tool to handle both conics and their transformations as elements of a particular Clifford algebra, \cite{HNV}. Also, and this is particularly exploited by our approach, intersections and contact points of conics may be obtained by simple operations in GAC, namely by the wedge and inner product with a great advantage that they can be expressed in the form of sums of products, \cite{b1}, and thus no numerical solver is needed. Within our algorithm, we use not only GAC but also its subalgebra called Compass Ruler Algebra (CRA), \cite{hild2}, which is a conformal model of two-dimensional Euclidean space with circles as intrinsic geometric primitives. 
	
	Our main idea is based on a recent concept of GAC, from \cite{HNV}, yet there have been some attempts to construct various models for conics, see e.g. \cite{hitzer}. Because the concept in \cite{HNV} is rather new, we refer to \cite{confit} for better insight into the object description and their manipulation in GAC. Generally, our main observation is that in GAC, you can express the four intersecting points of two specific conics without knowing their precise coordinates and, consequently, you may construct a degenerate conic of two intersecting lines containing the conics' intersections. The lines may be separated by means of standard linear algebra operations, you either continue with a geometric approach and calculate the intersection with a circle in CRA, see Example \ref{circle_case}, or you put the line equation together with a quadratic equation of on a conic to find their intersections analytically, see Example \ref{intersect_general}. 
	
	We shall demonstrate our results on examples of specific switched systems. We provide outputs of our implementation in Python using a module clifford for symbolic GAC operations.

	\section{State of the art}
	Nowadays, the most popular approach in searching for a control of a switched system is connected with Lyapunov functions, \cite{lyap}, which requires sophisticated  algorithmic structures and precise criteria for checking the validity of the used method in a particular problem. The ordinary Lyapunov function is used to test whether a given dynamical system is stable (more precisely asymptotically stable), but does not provide any information about controllability. Particular specificity of switched systems is in its interaction between the continuous variable and the discrete state, which is not present in the standard control systems. In the works of  D. Liberzon \cite{Lib1,Lib2}, the author addresses the problems of stability and control for particular types of switched systems, using the analytical approach, i.e. Lyapunov function and Brockett’s condition for asymptotic stability by continuous
	feedback and controllability. The stabilization problem for switched positive regular linear systems by state-dependent switching was considered in \cite{din} and an  anti-bump switching control problem was introduced in \cite{han}.

	The issue of an optimal control has also been addressed several times. The most popular are problems of  time- or distance-optimality. For example, finding time-optimal control for a dynamical system was considered by Nasir Uddin Ahmed, \cite{DSC}, while for switched systems, analogical problem was considered in \cite{Sei}, where the author constructs a minimizing sequence and uses compactness property for finding a subsequence that minimizes the cost functional.   Another way of optimisation is a construction of a switching path with minimal amount of switches which is of our particular interest. This is clearly achieved if the switch points are contact points, i.e. consequent ellipses are circumscribed. We are not aware of any similar construction in literature.

	We should also pay attention to the existing methods for these intersections: in a Euclidean space, the problem of finding intersection points of two conics is reduced to solving the system of quadratic equations using numerical methods.
	%

	\section{Geometric algebra}
	\label{sec:1}
	By geometric algebra(GA) we mean a Clifford algebra with a specific embedding of a Euclidean space (of arbitrary dimension) in such a way that the intrinsic geometric primitives as well as their transformations are viewed as elements of a single vector space, precisely multivectors. Concept was introduced by D. Hestenes in \cite{hestenes} and has been used in many mathematical and engineering applications since, see e.g. \cite{b7,colour}. 
	
	Great calculational advantage of GA is that geometric operations such as intersections, tangents, distances etc., are linear functions and therefore their calculation is efficient. To demonstrate this, we refer to \cite{Perwass} for the basics of geometric algebras, especially for conformal representation of a Euclidean space. Indeed, three-dimensional Euclidean space is represented in a Clifford algebra $\mathcal Cl(4,1),$ and the consequent geometric algebra is often denoted as $\mathbb{G}_{4,1}$ with spheres of all types as geometric primitives and Euclidean transformations at hand, see eg. \cite{Dorst}. In the sequel, we use also the two-dimensional subalgebra $\mathbb{G}_{3,1}$ called a Compass Ruler Algebra (CRA), \cite{hild2}, which is an analogue of $\mathbb{G}_{4,1}$ for two-dimensional Euclidean space.
	
	Let us now recall the generalisation of $\mathbb {G}_{4,1}$,i.e geometric algebra for conics(GAC), proposed by C. Perwass to generalize the concept of (two-dimensional) conformal geometric algebra $\mathbb{G}_{3,1}$, \cite{Perwass}. Let us stress that we use the notation of \cite{HNV}. In the usual basis $\bar{n},e_1,e_2,n$, the embedding of a plane into $\mathbb{G}_{3,1}$ is given by
	$$
	(x,y)\mapsto \bar{n} + xe_1+ye_2+\frac12(x^2+y^2)n,
	$$
	where $e_1,e_2$ form the Euclidean basis and $\bar n$ and $n$ stand for a specific linear combination of additional basis vectors $e_3, e_4$ with $e_3^2=1$ and $e_4^2=-1$, giving them the meaning of the coordinate origin and infinity, respectively, \cite{Perwass}. Hence the objects representable by vectors in $\mathbb{G}_{3,1}$ are linear combinations of $1, x, y, x^2+y^2$, i.e. circles, lines, point pairs and points. If we want to cover also general conics, we need to add two terms: $\frac12(x^2-y^2)$ and $xy$. It turns out that we need two new infinities for that and also their two corresponding counterparts (Witt pairs), \cite{Lounesto}. Thus the resulting dimension of the space generating the appropriate geometric algebra is eight.

	Analogously to CGA and to the notation in \cite{Perwass}, we denote the corresponding basis elements as
	\begin{align} \label{basis} \bar{n}_+,\bar{n}_-,\bar{n}_\times,e_1,e_2,n_+,n_-,n_\times.
	\end{align}
	This notation suggests that the basis elements $e_1,e_2$  play the usual role of standard basis of the plane while the null vectors $\bar{n}$, $n$ represent the origin and infinity, respectively. Note that there are three orthogonal `origins' $\bar{n}$ and three corresponding orthogonal `infinities' $n$.
	In terms of this basis,  a point of the plane ${\bf x} \in\mathbb{R}^2$ defined by ${\bf x}=xe_1+ye_2$ is embedded using the operator $\mathcal C: \mathbb{R}^2 \to \mathcal{C}one\subset \mathbb{R}^{5,3}$, which is defined by
	\begin{align} \label{embedding}
	\mathcal C(x,y)=\bar{n}_+ + xe_1+ye_2+\frac12(x^2+y^2)n_+ + \frac12(x^2-y^2)n_- + xyn_\times.
	\end{align}
	The image $\mathcal{C}one$ is an analogue of the conformal cone. In fact, it is a two-dimensional real projective variety determined by five homogeneous polynomials of degree one and two.
	\begin{definition}
		Geometric Algebra for Conics (GAC) is the Clifford algebra $\mathbb G_{5,3}$ together with the embedding  \eqref{embedding} in the basis \eqref{basis}.
	\end{definition}
	Note that, up to the last two terms, the embedding \eqref{embedding} is the embedding of the plane into the two-dimensional conformal geometric algebra $\mathbb{G}_{3,1}$. In particular, it is evident that the scalar product of two embedded points is the same as in $\mathbb{G}_{3,1}$, i.e. for two points ${\bf x},{\bf y}\in\mathbb{R}^2$ we have
	\begin{align} \label{distance}
	\mathcal C({\bf x})\cdot \mathcal C({\bf y})=-\frac12 \|{\bf x}-{\bf y}\|^2,
	\end{align}
	where the standard Euclidean norm is considered on the right hand side. This demonstrates the linearisation of distance problems. In particular, each point is represented by a null vector.
	Let us recall that the invertible algebra elements are called versors and they form a group, the Clifford group, and that conjugations with versors give transformations intrinsic to the  algebra. Namely, if the conjugation with a $\mathbb{G}_{5,3}$ versor $R$ preserves the set $\mathcal{C}one$, i.e. for each ${\bf x}\in\mathbb{R}^2$ there exists such a point $\bar{{\bf x}}\in\mathbb{R}^2$ that
	\begin{equation} \label{ts} R\mathcal C({\bf x})\tilde{R} = \mathcal C( \bar{{\bf x}}),\end{equation}
	where $\tilde{R}$ is the reverse of $R$, then ${\bf x}\to\bar{{\bf x}}$ induces a transformation $\mathbb{R}^2\to\mathbb{R}^2$ which is intrinsic to GAC. See \cite{HNV} to find that the conformal transformations are intrinsic to GAC.
	
	Let us also recall the outer (wedge) product, inner product and the duality \begin{equation}\label{duality}A^*=AI^{-1},\end{equation}
	where $I=e_{12345678}$ is a pseudoscalar, i.e the highest grade element
	For our purposes, we stress that these operations correspond to sums and products only. Indeed, the wedge product is calculated as the outer product of vectors on each vector space of the same grade blades, while the inner product acts on these spaces as the scalar product. The extension of both operations to general multivectors adds no computational complexity due to linearity of both operations. 
	Let us also recall that if a conic $C$ is seen as a wedge of five different points (which determine a conic uniquely), we call the appropriate 5-vector $E^*$ an outer product null space representation (OPNS) and its dual $E$, indeed a 1-vector, the inner product null space (IPNS) representation. The reason is that if a point $P$ lies on a conic $C$ then
	$$P\cdot E=0\quad \text{and}\quad P\wedge E^*=0.$$
	Duality calculation is given by \eqref{duality}. Consequently, intersections of two geometric primitives are given as the wedge product of their IPNS representations, i.e. 
	$$C_1\cap C_2=E_1\wedge E_2$$
	for two conics $C_1,C_2$ and their IPNS representations $E_1$ and $E_2$, respectively, see \cite{HNV}.
	
	Let us describe the inner product representation more precisely.
	An element $A_I\in\mathbb{G}_{5,3}$ is the inner product representation of a geometric entity $A$ in the plane if and only if $A=\{{\bf x}\in \mathbb{R}^2: \mathcal C({\bf x}) \cdot A_I=0\}$.
	The representable objects can be found by examining the inner product of a vector and an embedded point.
	A general vector in the conic space $\mathbb{R}^{5,3}$ in terms of our basis is of the form $$v=\bar{v}^+\bar{n}_++\bar{v}^-\bar{n}_-+\bar{v}^\times\bar{n}_\times+v^1e_1+v^2e_2+v^+n_++v^-n_-+v^\times n_\times$$
	and its inner product with an embedded point is then given by
	$$
	\mathcal C(x,y)\cdot v=-\frac12(\bar{v}^++\bar{v}^-)x^2-\bar{v}^\times xy-\frac12(\bar{v}^+-\bar{v}^-)y^2+v^1x+v^2y-v^+,
	$$
	i.e. by a general polynomial of degree two. Thus the objects representable in GAC are exactly conics.
	We also see that the two-dimensional subspace generated by infinities $n_-,n_\times$ is orthogonal to all embedded points.
	In other words, the inner representation of a conic in GAC can be defined as a vector 
	\begin{align} \label{conic} 
	Q_I=\bar{v}^+\bar{n}_++\bar{v}^-\bar{n}_-+\bar{v}^\times\bar{n}_\times+v^1e_1+v^2e_2+v^+n_+.
	\end{align}

	It is well known that the type of a given unknown conic can be read off its matrix representation, which in our case for a conic given by vector \eqref{conic}  reads
	\begin{align} \label{conic_matrix} 
	Q=\begin{pmatrix}-\tfrac12(\bar{v}^++\bar{v}^-) & -\tfrac12\bar{v}^\times&\frac12v^1\\
	-\tfrac12\bar{v}^\times & -\frac12(\bar{v}^+-\bar{v}^-) &\frac12v^2\\
	\frac12v^1 & \frac12v^2 & -v^+\end{pmatrix}.
	\end{align} 
	The entries of \eqref{conic_matrix} can be easily computed by means of the GAC inner product:
	\begin{align*}  
	q_{11}&=Q_I\cdot\tfrac12(n_+-n_-), & & q_{12}=q_{21}=Q_I\cdot\tfrac12 n_\times,\\
	q_{22}&=Q_I\cdot\tfrac12(n_++n_-), & & q_{13}=q_{31}=Q_I\cdot\tfrac12 e_1,\\
	q_{33}&=Q_I\cdot \bar{n}_+, & & q_{23}=q_{32}=Q_I\cdot\tfrac12 e_2.
	\end{align*} 
	It is also well known how to determine the internal parameters of a conic and its position and the orientation in the plane from the matrix \eqref{conic_matrix}. Hence all this can be determined from the GAC vector $Q_I$ by means of the inner product.
	%
	%
	%
	%
	%

	The classification of conics is well known. The non-degenerate conics are of three types, the ellipse, hyperbola, and parabola. Now, we present the vector form \eqref{conic} appropriate to the simplest case, i.e. an axis-aligned ellipse $E_I$ with its centre in the origin and semi-axis $a,b$. The corresponding GAC vector is of the form
	\begin{equation}\label{ellipse_axis_aligned_IPNS}
	E_I = (a^2+b^2)\bar{n}_++(a^2-b^2)\bar{n}_-a^2b^2 n_+.
	\end{equation}
	More generally, an ellipse $E$ with the semi-axis $a,b$ centred in $(u,v)\in\mathbb{R}^2$ rotated by angle $\theta$ is in the GAC inner product null space (IPNS) representation given by
	\begin{align}\label{ellipse_IPNS}
	E&=\bar{n}_+-(\alpha\cos2\theta)\bar{n}_- -(\alpha\sin2\theta)\bar{n}_\times \nonumber \\ 
	&\quad + (u+u\alpha\cos2\theta-v\alpha\sin2\theta)e_1
	+(v+v\alpha\cos2\theta-u\alpha\sin2\theta)e_2 \\ \nonumber
	&\quad +\tfrac12\left(u^2+v^2-\beta-(u^2-v^2)\alpha\cos2\theta-2uv\alpha\sin2\theta\right)n_+.
	\end{align}

	\noindent
	For proofs and further details see \cite{HNV}.
	\begin{remark}\label{CRA_line}
		Note that a line in GAC is not an intrinsic primitive and thus we may understand it as a CRA object, \cite{HNV}. Therefore its IPNS representation has the same form as in CRA, \cite{hild2}, 
		$$n_1e_1+n_2e_2+dn_+$$
		where $n=(n_1,n_2)$ is the normal vector and $d$ is the distance from coordinate origin. The OPNS representation of a line passing through two points $p_1$ and $P_2$ is, \cite{HNV} of the form
		$$P_1\wedge P_2\wedge n_+\wedge n_-\wedge n_\times.$$
	\end{remark}
	
	Regarding the transformations, our algorithm uses explicitly just isotropic scaling which is non-Euclidean but behaves in the very same way. Indeed, the scaling is generated by a bivector and the action of this bivector on a GAC element is given by conjugation. To show the precise coordinate form and all possibilities we recall the following proposition, \cite{HNV}. Note that the reason for three variants is again given by the fact that conics live in a six-dimensional subspace of eight-dimensional GAC 1-vector space, where in our notation two infinities $n_-$ and $n_+$ are redundant, and therefore there are three types of scalings associated to the respective infinities.

	\begin{prop}
		\label{scalor}
		The scalor for a scaling by $\alpha\in\mathbb R^+$ is given by $S=S_+S_-S_\times$, where
		\begin{equation*}
		S_+=\tfrac{\alpha+1}{2\sqrt{\alpha}}+\tfrac{\alpha-1}{2\sqrt{\alpha}} \bar{n}_+ \wedge n_+, \quad
		S_- =\tfrac{\alpha+1}{2\sqrt{\alpha}}+\tfrac{\alpha-1}{2\sqrt{\alpha}}  \bar{n}_- \wedge n_-,\quad
		S_\times=\tfrac{\alpha+1}{2\sqrt{\alpha}}+\tfrac{\alpha-1}{2\sqrt{\alpha}}  \bar{n}_\times \wedge n_\times.
		\end{equation*}
	\end{prop}
	\noindent For proof see \cite{HNV}.
	All transformations apply on a vector in GAC by conjugation \eqref{ts} of the appropriate versor formal exponential. This holds also for translations and rotations, for their precise form see \cite{HNV}.
	

	\section{Intersections and contact points in GAC} 
	In this section we provide a procedure for intersecting two conics, particularly ellipses with a common centre in the coordinate origin but in a general mutual position otherwise. Moreover, we consider a system of circumscribed ellipses as in  and show a procedure for detecting the first order contact points, i.e. the \ref{setting_inters} points where the ellipses touch with identical first order derivative. Again, the contribution of GAC lies in avoiding the use of a solver which leads to accuracy improvement.
	
	Let us first describe some differences to CRA or its 3-dimensional version CGA (Conformal Geometric Algebra). Crucial difference lies in the type of objects that are intrinsic to respective structures. For CRA (CGA), spheres (circles) are the geometric primitives that may be represented by specific elements. Taking into account that lines and planes are spheres with infinite radii and a point pair is a 1-dimensional sphere, we receive all geometric primitives for Euclidean geometry. Moreover, their intersection still remain such objects, indeed, an intersection of two spheres or two circles are circles or point pairs, respectively. Therefore, intersections that are realised by wedge of IPNS representations remain representatives of Euclidean primitives intrinsic to CRA (CGA). On contrary, in GAC the situation is different. Even if we restrict to the case of co-centric ellipses, their intersection is a "four point" which has no meaning in the sense of conic-sections. Indeed, a planar conic is uniquely generated by five points at least. This leads to an algorithm that may be used for co-centric conics (all types). On the other hand, the algorithm is still geometric-based and may be realised by a sequence of simple operations in GAC, i.e. there is no numerical solver involved.
	
	\subsection{Intersections}
	\label{sect_intersections}
	Let us now present the procedure for getting intersections of two co-centric ellipses, i.e. the set up according to Figure \ref{setting_inters}. Note that we may assume, without loss of generality, that the ellipses have four points of intersections. Other cases are not of our interest and would be recognized by the form of GAC element representing no conic. Furthermore we may assume that the ellipse centres are situated at the coordinate origin, otherwise the whole picture may be translated in GAC to fulfil this assumption. 
	
	\begin{figure}[h]
		\centering
		\includegraphics[width=70mm]{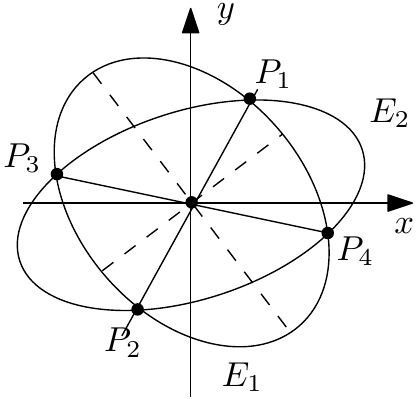}
		\caption{Initial setting of the intersection problem}
		\label{setting_inters}
	\end{figure}

	We start by taking two IPNS representations of ellipses $E_1$ and $E_2$ and wedge them. The result corresponds to the common points of both geometric primitives. This is standard operation intrinsic to any geometric algebra. In our case, we receive an IPNS representative of a four-point $E_1\wedge E_2=P_1\wedge P_2\wedge P_3\wedge P_4$ as in Figure \ref{setting_inters}. 
	
	Therefore, as the next step, we construct a degenerate conic, more precisely a pair of intersecting lines $(E_1\wedge E_2)^*\wedge \overline n_+$, where $\overline n_+$ represents origin of the Euclidean coordinates and therefore the common ellipse centre, and $(E_1\wedge E_2)^*$ is the four-point's OPNS representation. 
	
	
	Now we need into decompose the pair of lines to two separate single lines. First, we construct the matrix $Q$ of its quadratic form by \eqref{conic_matrix}. 
	Note that $Q$ is a symmetric singular matrix. To decompose a degenerate conic we follow an algorithm described in \cite{geom}. We recall the algorithm just to present that all operations involved are sums and products in the form of determinant calculations. The only numerical inaccuracy may be imported by a square root calculation.
	
	Indeed, to decompose a pair of intersecting lines into two distinct lines we have to find a skew-symmetric matrix $P$ formed by parameters $\lambda$, $\mu$, and $\tau$ such that $N=Q + P$	is of rank 1. Thus in our case we have 
	\begin{align} 
	N=	\begin{pmatrix}q_{11} & q_{12}&q_{13}\\
	q_{21}& q_{22} &q_{23}\\
	q_{31} & q_{32} & q_{33}\end{pmatrix}+	\begin{pmatrix} 0 &\tau& -\mu\\-\tau & 0& -\lambda\\ \mu & \lambda & 0 \end{pmatrix}.
	\label{matrix_decompose}
	\end{align} 
	The rank condition reads that every $2 \times 2$ submatrix determinant must vanish.
	Thus the necessary conditions for the parameters $\lambda$, $\mu$, and $\tau$ are:
	\begin{align*}
	\tau^2=-\begin{vmatrix}q_{11} & q_{12}\\q_{21}& q_{22} \end{vmatrix}, \quad \mu^2=-\begin{vmatrix}q_{11} & q_{13}\\q_{31}& q_{33} \end{vmatrix}, \quad \lambda^2=-\begin{vmatrix}q_{22} & q_{23}\\q_{32}& q_{33} \end{vmatrix}.
	\end{align*}
	This determines the parameters $\lambda$, $\tau$, and $\mu$ up to their sign.
	%
	In general case, to get precise values of $\lambda$, $\mu$, and $\tau$, one can take a nonzero column of the matrix dual to $Q$ and divide it with a specific factor, see \cite{geom}. In the case that the lines are passing through the origin, the division may be omitted and thus only the dual matrix, i.e. nine determinants of order 2, have to be calculated, see \cite{geom}.
	
	By taking an arbitrary nonzero row and a nonzero column in the matrix $N$ we get the coefficients of the respective separated lines. We shall now recall that a single line represents no conic and therefore it is not a geometric primitive intrinsic to GAC. Yet it is understood as an element of subalgebra CRA, i.e. a model of two-dimensional Euclidean space formed by a Clifford algebra $\mathcal{C}l(3,1),$ see \cite{hild2}.

	\begin{example}	\label{intersect_general}
		To construct an ellipse, we need the semi-axis lengths $a,b$, centre coordinates $c_1,c_2$ and the angle of rotation $\theta$, \eqref{ellipse_IPNS}. Let us consider two ellipses $Ell1$ and $Ell2$ with parameters $(a,b,c_1,c_2,\theta)=(2,4,0,0,0)$ and $(4,2,0,0,\frac{\pi}{6})$, respectively, see Figure \ref{example1}.
		\begin{figure}
			\begin{subfigure}[h]{0.48\textwidth}
				\centering
				\includegraphics[scale=.5]{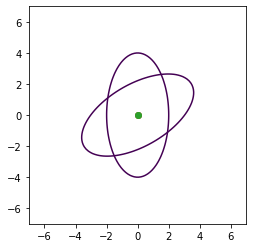}
				\caption{Ellipse setting}
				\label{example1}
			\end{subfigure}
			\begin{subfigure}[h]{0.4\textwidth}
				\centering
				\includegraphics[scale=.5]{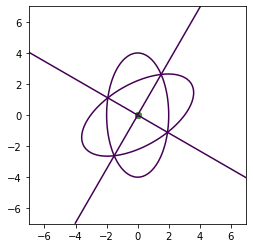}
				\caption{Pair of lines}
				\label{example2}
			\end{subfigure}
			\caption{Setting of Example \ref{intersect_general}}
		\end{figure}
		Their IPNS representations will then be of the form
		$$Ell1=\bar n_+ +\frac{3}{5}\bar n_- -\frac{16}{5}n_+$$
		and
		$$Ell2=\bar n_+ -\frac{3}{10}\bar n_- -\frac{3\sqrt 3}{10}\bar n_\times -\frac{16}{5}n_+.$$
		%
		\noindent		
		If transformed to OPNS, they become four-vector, therefore their representation corresponds to a wedge of four GAC points. By wedging the origin represented by $\bar n_+,$ we receive an OPNS representation of a degenerate conic, more precisely of a pair of intersecting lines. Their IPNS form is 
		$$-\frac{72}{25}\bar n_- -\frac{24\sqrt 3}{25}\bar n_\times.$$
		The type of the conic may be easily checked using their matrix form
		$$\begin{pmatrix}
		\frac{36}{25} & \phantom{-}\frac{12\sqrt 3}{25} & 0\\
		\frac{12\sqrt 3}{25} & -\frac{36}{25} &  0\\
		0 & \phantom{-}0 &0
		\end{pmatrix}. $$
		After normalization, the equation of this conic is
		$x^2-y^2+\frac{24\sqrt 3}{3}xy=0.$
		Thus we have a pair of lines containing all four ellipses' intersections and the origin, see Figure \ref{example2}.

		Let us provide all necessary inputs for procedure of line separation \eqref{matrix_decompose} in the same form:
		$$ \begin{pmatrix}
		\frac{36}{25} & \phantom{-}\frac{12\sqrt 3}{25} & 0\\
		\frac{12\sqrt 3}{25} & -\frac{36}{25} &  0\\
		0 & \phantom{-}0 &0
		\end{pmatrix}
		+
		\begin{pmatrix}
		0 & -\frac{12\sqrt 3}{25} & 0\\
		\frac{36\sqrt 3}{25} & -\frac{36}{25} &  0\\
		0 & \phantom{-}0 &0
		\end{pmatrix},$$
		i.e.
		$\mu=0, \quad \tau=-\frac{24\sqrt 3}{25}, \quad\lambda=0.$
		Therefore the pair of lines' matrix is of the form
		$$\begin{pmatrix}
		\frac{36}{25} & -\frac{12\sqrt 3}{25} & 0\\
		\frac{36\sqrt 3}{25} & -\frac{36}{25} &  0\\
		0 & \phantom{-}0 &0
		\end{pmatrix}$$
		and thus the form of separated lines may be easily derived according to the first (nonzero) row and column. After normalization we receive
		$$\frac{1}{2}x+\frac{\sqrt 3}{2}y=0, \quad \mathrm{and}\quad \frac{\sqrt 3}{2}x-\frac{1}{2}y=0.$$
		It is clear that they are perpendicular which has been expected due to settings' symmetries.
		%
		%

		
		Thus we get a system of quadratic (Ellipse) and linear (Line) equations, which is a reduction to  one quadratic equation, that does not require the use of the solver.
		In our case 
		\begin{center}
			$\begin{cases}
			x + \sqrt{3} y = 0,  \\
			\frac{x^2}{4} + \frac{y^2}{16} = 1,
			\end{cases}$\\
		\end{center}
		for $x=-\sqrt{3}y$,
		we get $13y^2-16=0$
		and	$y=\pm \frac{4\sqrt{13}}{13} $,
		$x=\mp \frac{4\sqrt{39}}{13}. $ 
		So we get intersection points $[\frac{4\sqrt{39}}{13},-\frac{4\sqrt{13}}{13}]$ and $[-\frac{4\sqrt{39}}{13},\frac{4\sqrt{13}}{13}]$.
		In the same way, by using the line  $	x + \sqrt{3} y = 0$, we get the following points:
		$[\frac{4\sqrt{39}}{13},\frac{4\sqrt{13}}{13}]$ and $[-\frac{4\sqrt{39}}{13},-\frac{4\sqrt{13}}{13}]$.
	\end{example}
	
	The described procedure can be applied to different types of the co-centred conics with 4 intersection points. Note that in GAC it is necessary to work with very high precision, because even small inaccuracy in coordinates of points to be wedged may lead to a different type of conic.
	\begin{example}
		Figure \ref{Python_general} demonstrates the output of the Python code for different types of the co-centred conics with 4 intersection points. It shows that our considerations are valid not only for ellipses but for an arbitrary pair of non-degenerate co-centered conics, see \cite{Loucka} for proofs.
		
		\begin{figure}
			\begin{subfigure}[h]{0.48\textwidth}
				\centering
				\includegraphics[scale=.65]{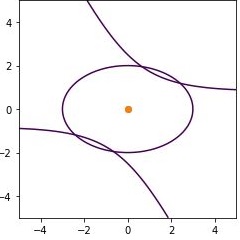}
				\caption{Ellipse and Hyperbola}
				
			\end{subfigure}
			\begin{subfigure}[h]{0.4\textwidth}
				\centering
				\includegraphics[scale=.65]{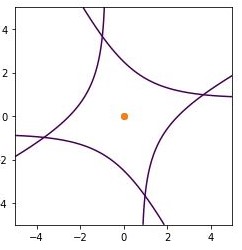}
				\caption{Hyperbola and Hyperbola}
				
			\end{subfigure}
			\caption{Non-degenerate co-centered conics}
			\label{Python_general}
		\end{figure}
		
	\end{example}
	
	\begin{example}\label{circle_case}
		In the case of axis-aligned ellipses we can apply a more geometric approach, \cite{BDVHS}. Given two ellipses $Ell1$ and $Ell2$ with parameters $(a,b,c_1,c_2,\theta)=(2,4,0,0,0)$ and $(4,2,0,0,0)$, respectively, we determine their IPNS representations according to \eqref{ellipse_axis_aligned_IPNS} in the form
		$$Ell1=\bar n_+ +\frac{3}{5}\bar n_- -\frac{16}{5}n_+$$
		and
		$$Ell1=\bar n_+ -\frac{3}{5}\bar n_- -\frac{16}{5}n_+,$$
		the intersecting points form a circle $C$ that may be constructed by $(Ell1\wedge Ell2)^*\wedge \bar n_+, $ \cite{BDVHS}, and thus represented by an element
		$$C=\frac{6}{5}\bar n_+ -\frac{96}{5}n_+,$$
		i.e. its standard equation will be
		$$x^2+y^2-\frac{32}{5}=0.$$
		Then we can construct a pair of intersecting lines $(Ell1\wedge Ell2)^*\wedge \bar n_+$ with IPNS representation
		$$\frac{6}{5}\bar n_- -\frac{96}{25}n_+,$$
		i.e. of the equation (after normalization)
		$-x^2+y^2=0.$
		
		The line decomposition procedure, although not necessary in this particular case, will lead to a pair of lines
		$y=x$ and $y=-x.$
		As CRA elements they are of the form
		$l_1=-\frac{\sqrt2}{2}e_1+\frac{\sqrt2}{2}e_2$ and $l_2=\frac{\sqrt2}{2}e_1+\frac{\sqrt2}{2}e_2,$
		respectively. Then it is enough to calculate the intersection $C\wedge l_1$ and $C\wedge l_2$ to get two point pairs $P_1,P_2$ in CRA. Consequently, a procedure for a point pair decomposition must be applied in the form
		\begin{align*}
		p_{i1} &=  \frac{- \sqrt{P_i \cdot P_i}+P_i}{n_+ \cdot P_i}, \: \: \:
		p_{i2} =   \frac{\sqrt{P_i \cdot P_i}+P_i}{n_+ \cdot P_i}\quad \mathrm{for}\quad i=1,2. \\
		\end{align*}
		In this very simple case we receive CRA points
		$$\bar n_+\pm\frac{4\sqrt 5}{5}e_1\pm\frac{4\sqrt 5}{5}e_2+\frac{16}{5}n^+,$$
		which means that the points of intersections are of the form $[\pm\frac{4\sqrt 5}{5},\pm\frac{4\sqrt 5}{5}].$
	\end{example}

	\subsection{Contact points}
	
	As mentioned above, by contact points we understand first order contact points, i.e. points where two curves have identical first order derivative. We shall describe how to receive a set of contact points for a given system of co-centred ellipses. Such system is formed as in Figure \ref{contact_pts} beginning with two intersecting ellipses $E_1$ and $E_2$. Then an ellipse $E'_2$ is constructed from $E_2$ just by scaling in such a way that $E'_2$ is circumscribed to $E_1$, i.e. they have two contact points. Then an ellipse $E'_1$ would be constructed from $E_1$ such that it would be circumscribed to $E'_2$ etc.  
	\begin{figure}[h]
		\centering
		\includegraphics[width=50mm]{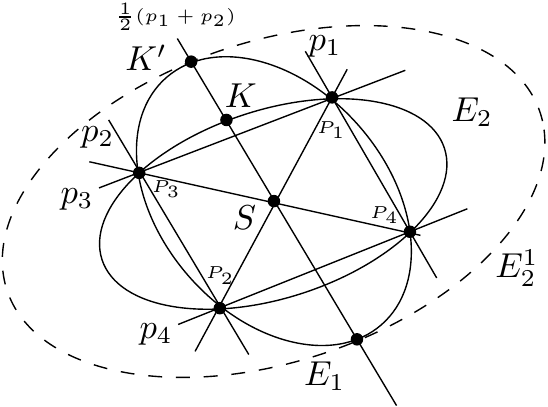}\hspace{10mm}
		\includegraphics[]{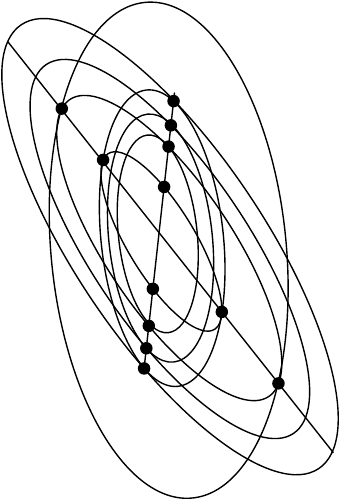}		
		\caption{Initial setting of contact points problem}
		\label{contact_pts}
	\end{figure}
	
	\begin{prop}
		Given a system of co-centric ellipses as in Figure \ref{contact_pts}, the contact points form a pair of intersecting lines. Furthermore, keeping the notation of \ref{contact_pts}, one of these lines is the axis of lines $p_1$ and $p_2$ denoted as $\frac{1}{2}(p_1+p_2)$ and, similarly, the other line is the axis of the lines $p_3$ and $p_4$ in Figure \ref{contact_pts}. 
	\end{prop}
	
	\begin{proof}
		Taking into account that the ellipses are co-centric and their symmetry properties, it is obvious that 4 points of their intersection form a parallelogram with diagonals passing through the common centre. The line $\frac{1}{2}(p_1+p_2)$ is the middle line of the parallelogram and passes through the common centre $S$ of the ellipses. Note that the notation $\frac{1}{2}(p_1+p_2)$ for the axis of $p_1$ and $p_2$ is exactly the way to calculate this line in CRA. Indeed, this is true for IPNS representations of $p_1$ and $p_2$.
	\end{proof}
	
	\begin{prop}\label{scale_factor}
		Given a system of co-centric ellipses as in Figure \ref{contact_pts}, a scalor transforming an ellipse $E_2$ to $E'_2$ may be calculated as $SP = \frac{|SK'|}{|SK|}$, where $K$ and $K'$ are the intersection points of the ellipses $E_2$ and $E_1$ with the line $\frac{1}{2}(p_1+p_2)$ and S is the common centre of the ellipses.
	\end{prop}
	
	\begin{proof}
		Due to the fact that $K'$ is the contact point of ellipses $E'_2$ and $E_1$, the scaling ellipse $E_2$ until the contact with $E_1$ means scaling the length section $SK$ up to the length of $SK'$, therefore  $SP = \frac{|SK'|}{|SK|}$. 
	\end{proof}
	
	\begin{remark}
		The transformation of $E_2$ to $E'_2$ is then given in GAC by a scalor according to Proposition \ref{scalor}. In the case of co-centered ellipses with the centre in the coordinate system origin we may just multiply the semi-axis lengths by the scaling factor.	
	\end{remark}	
	
	\section{Controllable $2\times2$ linear switched dynamical systems}\label{sec:2}
	
	Let us briefly recall the basic terminology in the switched systems theory. By a switched system we mean the following system	
	$$\dot{\boldsymbol{x}}(t) = f_{\sigma(t)}(\boldsymbol{x}(t)), \quad \boldsymbol{x}(0) = \boldsymbol{x}_0$$
	where $\boldsymbol{x} \in \mathbb{R}^m$ is called \textit {a continuous state}, $\sigma$ stands for \textit{a discrete state} with values from an index set $M : = \{1,\ldots,n\}$, and $f_{\sigma(t)}$, for $\sigma(t) \in M$, are given vector fields. 
	
	The behaviour of the dynamical system is regulated by the switching signal. Namely, at specific time moments, i.e. for $t=\tau_1,\dots,\tau_n,$ the system changes its setting from $\sigma(\tau_i)$ to $\sigma(\tau_{i+1})$, hence the trajectory of the system, starting from $t=\tau_i$, is given by the vector field $f_{\sigma(\tau_{i+1})}$ instead of $f_{\sigma(\tau_i)}$.
	In the works on switched systems, switching times can be random or given by some law. 
	In the sequel, we consider a different formulation of the problem, i.e. the switching signal is under our control.
	To guarantee that there exists a path connecting two arbitrary points, let us recall the following definition.
	\begin{definition}
		We say that the switched system   $$\dot{\boldsymbol{x}}(t) = f_{\sigma(t)}(\boldsymbol{x}(t)), \quad \boldsymbol{x}(0) = \boldsymbol{x}_0$$ is \textit{controllable} if 
		for any two points $A,B$ from the state space there exists a switching signal generating a continuous path from $A$ to $B$.
		
	\end{definition}
	
	The above definition corresponds to the concept of controllability for control systems of the form 
	$$\dot{\boldsymbol{x}} = f(\boldsymbol{x},\boldsymbol{u}), \quad \boldsymbol{x}(0) = \boldsymbol{x}_0,$$
	where the control $\boldsymbol{u}(t)$ plays the role of a switching signal.
	
	Particularly, linear switched systems, \cite{Col},  of the form $$\dot{\boldsymbol{x}}(t) = A_{\sigma(t)}\boldsymbol{x}(t), \quad \boldsymbol{x}(0) = \boldsymbol{x}_0\neq 0,$$
	where $A_1\dots A_n$ are given matrices, are of our interest. 
	
	More precisely, the case of $2\times 2$ matrices with both subsystems having pure imaginary eigenvalues is studied. This case has already been considered in \cite{Der}, and the main difference lays in using GAC as a suitable space for geometric operations with the ellipses, for elementary notions see \Cref{sec:1}. 
	
	First, consider the problem of oscillation of a spring pendulum under the condition of absence of external and friction forces    
	$$\ddot{\boldsymbol{x}}=-k\boldsymbol{x},$$ 
	with a switchable stiffness coefficient  $k>0$, that changes value from
	$k_1$ by joining and removing an additional spring  with a stiffness coefficient $k_2$. 
	Two cases can be considered. If the springs are connected in parallel, the parameter $k$ of the system switches between $k=k_1$ and  $k=k_1+k_2$. If the connection is series, the parameter $k$ of the system switches between $k=k_1$ and $k=\frac{k_1 k_2}{k_1+k_2}$.

	Let us rewrite the differential equation of the pendulum oscillations as a switched system:
	$${\dot{\boldsymbol{x}}(t) = A_i \boldsymbol{x}(t)}, \quad A_i\in\mathcal{M}at_2(\mathbb R), \quad i=1,2.
	$$
	Without the loss of generality, let us assume that we start and end with the first system $i=1$. 
	Suppose that two nonzero points (initial $A(x_1,y_1)$ and final $B(x_2,y_2)$) are given.

	By rewriting the coordinates of  $\boldsymbol{x}$ as $(x,y)$ we get that the solutions of the system $${\dot{\boldsymbol{x}}(t) = A_1(\boldsymbol{x}(t)),}
	$$
	can be determined in the form
	
	\begin{equation*}
	\begin{array}{l}
	x(t)=\gamma_1 \sin(\sqrt{\alpha_1} t)+\gamma_2 \cos(\sqrt{\alpha_1}t),\\
	y(t)=\sqrt{\alpha_1}\gamma_1 \sin(\sqrt{\alpha_1} t)-\sqrt{\alpha_1}\gamma_2 \cos(\sqrt{\alpha_1}t),\quad \gamma_1, \gamma_2\in\mathbb R.\end{array}
	\label{xy1} \end{equation*}
	The trajectories for the system with pure imaginary eigenvalues are given by ellipses. In the case $\mathrm{Tr}A_i=0$ , i.e. the case of the spring pendulum without damping, for example, $$A_i = \begin{bmatrix}
	0 & 1 \\
	-\alpha_i & 0
	\end{bmatrix},\quad \alpha_i\in\mathbb R^+,
	$$
	we have an axis-aligned ellipse, while if $\mathrm{Tr}A_i\neq0$, then the ellipses are rotated and the given switched system is equivalent to the equation describing the oscillatory system with damping. In this case the rotation angle can be calculated by means of the conic matrix \eqref{conic_matrix} in terms of geometric algebra:
	$$\theta=\begin{cases}
	\arctan\left(\frac{1}{q_{12}}\left(q_{22}-q_{11}-\sqrt{(q_{11}-q_{22})^2+q_{12}}\right)\right),\quad  q_{12}\neq0\\
	0,\quad \,  q_{12}=0,\quad q_{11}<q_{22}\\
	\frac{\pi}{2},\quad q_{12}=0,\quad q_{11}>q_{22}
	\end{cases}$$
	

	
	\section{Algorithm for a switching path construction}\label{sec:3}
	
	In the following, we describe the algorithm for finding a control of a switched system, i.e. finding a path composed of the systems' integral curves from the initial point $A$ to the endpoint $B$ such that the number of switches is minimal. Consider the case $n= 2$, i.e. only two systems are included, and both starting and final ellipse belong to the same family. To apply the GAC based calculations, it is necessary to get the exact GAC form of the representatives of both families of ellipses. Thus the system of ODEs is solved numerically (e.g. by Runge-Kutta method) with the initial condition at the starting point $A$. This will give us a set of points representing the initial ellipse. After applying the GAC conic fitting algorithm, \cite{HNV}, we get the ellipse in IPNS representation. Note that according to \cite{Loucka}, the algorithm may be further specified by prescribing the resulting ellipse to be axis-aligned and with its centre placed in the origin. This makes the initial trajectories very precise.  
	\begin{enumerate}
		\item Get $ A,B$, the starting and final point, respectively, i.e. get their conformal embedding $\mathcal C(A), \mathcal C(B)$ to GAC, \eqref{embedding}.
		\item Find the IPNS representation of the initial ellipse $E_1^1$ by conic fitting algorithm. Let us denote its semiaxis by $a$ and $b.$
		\item Find the final ellipse $E_f$ in the following two steps:
		\begin{itemize}
			\item[$\bullet$] Construct a line $l$ passing through the points $\mathcal C([0,0])=\bar n^+$ and $\mathcal C(B)$ according to Remark \ref{CRA_line}: $$l = \mathcal C(B)\wedge n_+ \wedge \bar{n}_+ \wedge n_- \wedge n_\times.$$ Find the intersection point $C = E_1^1 \cap l$ of the line and the starting ellipse, i.e. solve a quadratic equation in a Euclidean space as in Example \ref{intersect_general}.
			\item[$\bullet$] According  to Proposition \ref{scale_factor}, the scale parameter between the starting and final ellipse is $${{SP}_1=\frac{|\bar{n}_+\cdot \mathcal C(B)|}{|\bar{n}_+\cdot \mathcal C(C)|}}.$$ Construct the scalor according to Proposition \ref{scalor} and the final ellipse $E_f$ by \eqref{ts} as $$E_f = S_+ S_- S_\times E_1^1 \bar{S}_\times \bar{S}_- \bar{S}_+.$$ 
		\end{itemize}
		\item Find the first intermediate ellipse $E_2^1$. Note that lower index shows the number of subsystem, to which ellipse belongs.
		Take e.g. $[0,b]$ as initial condition and find IPNS representation of the sample ellipse $E_s$ by GAC conic fitting algorithm, \cite{Loucka}. In order to get the circumscribed ellipse $E_2^1,$ we need $E_s$ to have four intersection points with $E_1^1$. This can be checked easily by determining the type of the conic $(E_1^1\wedge E_s)^*\wedge \bar n_+$ and we shall scale $E_s$ by a factor $\alpha <1$ as in \eqref{ts} until the conic type of $(E_1^1\wedge E_s)^*\wedge \bar n_+$ is two intersecting lines. Then continue. 
		\begin{itemize}
			\item[$\bullet$] Find the intersections $E_1^1\wedge E_2^1$ according to Section \ref{sect_intersections}.
			\item[$\bullet$] Construct pair of lines $p_1$ and $p_2$ according to Remark \ref{CRA_line} and calculate their axis $p=\frac{1}{2}(p_1+p_2)$. To recognize the correct line one can use the inner product with the lines determined by the ellipse $E_1^1$ semiaxis denoted also as $a$ and $b$. Indeed, $a\cdot p \leq b\cdot p$ which is clear from Figure \ref{contact_pts} and from the properties of inner product similar to the scalar product of vectors. Clearly, in IPNS representation both $a$ and $p$ are 1-vectors.
			\item[$\bullet$] Construct the intersection of $p$ and  $E_1^1$ by $P_{t12}=E_1\wedge p$. Then $P_{t12}$ is a point pair of contact points $P_{t1}$ and $P_{t2}$.
			\item[$\bullet$] Calculate the scaling parameter $\alpha$ between the ellipses $E_s$ and $E_2^1$ as $\alpha=\frac{\parallel SP_{t1}\parallel}{a'},$ where $a'$ is the length of $E_s$ semiaxis. Note that the ellipse parameters may be easily computed from the matrix \eqref{conic_matrix}. Correctness of this calculation follows from Proposition \ref{scale_factor}.
			\item[$\bullet$] Construct $E_2^1$ by rescaling $E_s$.
		\end{itemize}
		\item Check the intersection between $E_f$ and $E_2^i$, where $i=1,2..$ is the number of additional ellipse in the following steps. 
		\begin{enumerate}
			\item If $E_f \cap E_2^i \neq \emptyset \implies$ find the intersection points of all ellipses, get the path from $A$ to $B$ by choosing the nearest point with respect to the path evolution. This will switch to final ellipse.
			\item If $E_f \cap E_2^i = \emptyset \implies$ calculate the scaling parameter $SP$ according to Proposition \ref{scale_factor}.
		\end{enumerate}
		By scaling $E_1^i, E_2^i$ using $SP$ get new pair of circumscribed ellipses	\begin{equation*}
		E_2^{i+1} := \text{scale}(E_2^i,SP),\quad
		E_1^{i+1} := \text{scale}(E_1^i,SP), 
		\end{equation*}
		Get intersection $E_1^{i+1} \cap E_2^{i+1}$, add points of intersection to the list of switching points and return to the beginning of the step 5  until $E_f\cap E_2^{j} \neq \emptyset$ for some $j$. This cycle constructs the sequence of ellipses from the starting ellipse to the final one.
		
	\end{enumerate}
	
	As a result, the above algorithm provides a sequence of switching points as well as a sequence of trajectories in GAC. For example of the resulting path see Figure \ref{fig:ex2}.
	Consider the following examples, which generalize the system from \cite[p. 6]{Cop}.

	%
	%
	%
	%
	%
	
	\begin{example}\label{ex:axis_aligned}
		Oscillatory system without damping.
		Consider the switched system $\dot{\boldsymbol{x}} = A_i \boldsymbol{x}$ where $A_1 = \begin{pmatrix}
		0 & 1\\
		-2 & 0
		\end{pmatrix},
		A_2 = \begin{pmatrix}
		0 & 1\\
		-\frac{1}{2} & 0
		\end{pmatrix}$. The initial point is $(2,5)$, and we need to find the way to the point $(12,22)$
		The set of used ellipses can be seen in Figure \ref{fig:ex1} while the set of switching points is
		$\{(0;-5,74456),(8,12404;0),(0;11,48913),(-16.24808;0),(0;-22,97825),\\(23,2054167141;12.86501593890354)\}.$ This is a result of Python code written in a module Clifford according to the algorithm in Section \ref{sec:3}. Note that the red points in Figure \ref{fig:ex1}, left, form the set of pints generated by Runge-Kutta method and you can see the fitted conic, too. 
		\begin{figure}[h]
			\includegraphics[width=68mm]{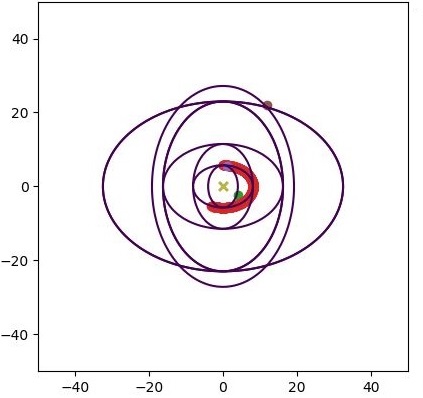}
			\includegraphics[width=68mm]{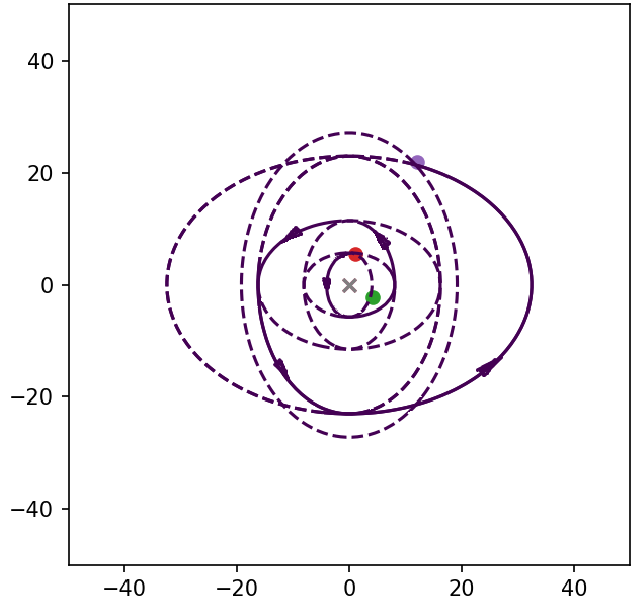}
			
			\caption{Example \ref{ex:axis_aligned}}
			\label{fig:ex1}
		\end{figure}

		%
		%
		
	\end{example}
	\begin{example}\label{ex:general} 
		Consider the switched system $\dot{\boldsymbol{x}} = A_i \boldsymbol{x}$ where $A_1 = \begin{pmatrix}
		0 & 1\\
		-2 & 0
		\end{pmatrix},\\
		A_2 = \begin{pmatrix}
		1 & 1\\
		-2 & 1
		\end{pmatrix}$. The initial point is $(2,5)$, and we need to find a path to the point $(30,22).$
		Both of the matrices have pure imaginary eigenvalues, so the system is switching between ellipses. Ellipses of the second family are rotated. That means that the second subsystem describes one of the cases of the oscillatory system with damping. These can also produce other types of conics, e.g. spirals, but that case is not the point of our interest The set of switching points is $\{(-2,88653912573;-3,697011208397),\\		(-4.23042514649;8.76807759024), (5,98270815467; -7,662511450902),\\ (-8,76807759024;18,1729216252),
		(-15,88149952097;-12,39990012429),\\ (-11,14111993673;43,01897176660)\}$ and the switching path calculated in Python module Clifford can be seen in Figure \ref{fig:ex2}.
		\begin{figure}[h]
			\centering
			\includegraphics[width=70mm]{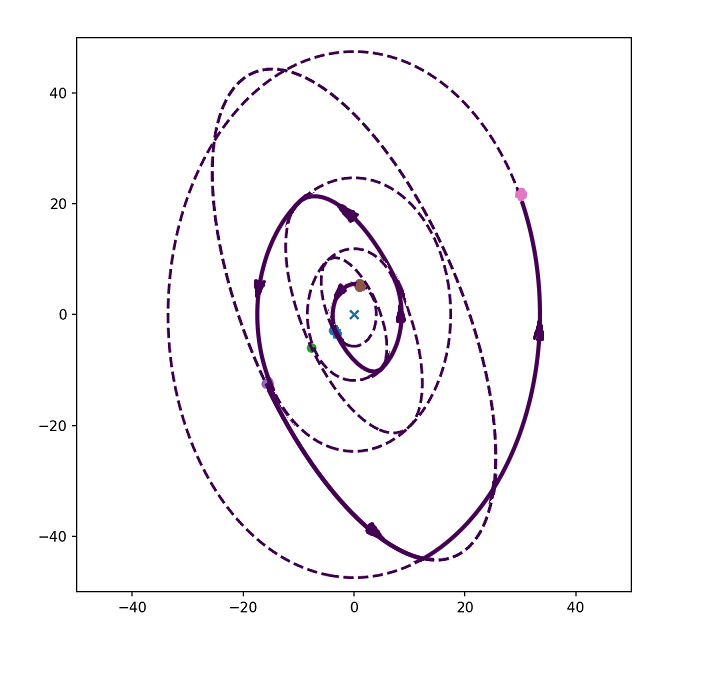}
			
			\caption{Example 4}
			\label{fig:ex2}
		\end{figure}
		
		%
		%
		
	\end{example}
	
	
	%
	%
	%
	%
	%
	%
	%
	
	
	\section{Conclusion and discussion}
	
	We have proposed a novel algorithm for optimal control of a switched dynamical systems with purely imaginary eigenvalues of both matrices. Namely, we constructed a switching path consisting of circumscribed ellipses with switching points at exactly contact points, which guaranteed minimal number of switching points. The whole construction has been implemented in Python module Clifford using its standard commands and functions for Geometric Algebra for Conics (GAC). Main advantage of GAC consists in a possibility of simple representation of transformed objects (for example rotated and scaled ellipses) together with a possibility of effective circumscribed ellipse construction. We stress that our geometric approach eliminates the need for any solver and therefore it minimizes numerical errors. 
	
	We demonstrated a complete geometric procedure for two families of axis-aligned, ellipses in Examples \ref{circle_case} and \ref{ex:axis_aligned}, where we demonstrated symbolic and Python calculations, respectively. In addition, we used a property of GAC that it contains a two-dimensional conformal geometric algebra CRA, where our calculations were completed. This case corresponds to an oscillatory switched system without damping. But our approach applies also for damped systems where the integral curves are formed by rotated ellipses, i.e. non axis-aligned, which we demonstrated in Examples \ref{intersect_general} and \ref{ex:general}. Also in this case no solver was needed because, in the system of two quadratic equations describing the ellipses' intersections, we replaced an ellipse equation by a line equation which reduced the degree and allowed analytic solution. Note that both approaches exploit the elegance of conics' manipulation in GAC by constructing a pair of lines containing the intersecting points and circumscribed ellipses simply calculated by GAC scaling with a factor determined according to Proposition  \ref{scale_factor}. 
	
	Let us point out that even the preparation of initial trajectories is highly geometric. Fitting a conic with prescribed properties in GAC eliminates an error in numerical solution to our switched system. Indeed, all trajectories will be precisely of a given type, i.e. co-centred and axis-aligned. Consequent GAC transformations do not change these properties and do not input any numerical errors. Indeed, the only place for a rounding error is the calculation of fractions and square roots because all operations in GAC may be converted to sums of products, \cite{BDVHS,b1}.
	
	Our algorithm generates a switching path that is optimal with respect to a number of switching points. Indeed, by constructing circumscribed ellipses we minimize the number of trajectories involved and therefore the number of switching points. Finally let us note that we applied our algorithm on systems that are controllable, \cite{Der}, i.e.  existence of a trajectory connecting the initial and final points is guaranteed.  
	%
	%
	
	%

	
	\section*{Acknowledgement}
	The research was supported by a grant no. FSI-S-20-6187.
	
	\bibliographystyle{siam} 
	\bibliography{sws_plain.bib}

	\noindent
	ANNA DEREVIANKO, PETR VA\v S\'IK 
	\newline
	Institute of Mathematics
	\newline 
	Brno University of Technology
	\newline
	Faculty of Mechanical Engineering, 
	\newline
	Technick\'a 2
	\newline
	616 69 Brno, Czech Republic
	\newline
	\noindent
	e-mail:{\tt \ derev.anna.08.96@gmail.com,\ Petr.Vasik@vutbr.cz }

	
	%
	%
	%
	%
\end{document}